\documentclass{amsart}
\usepackage{amsmath}%
\usepackage{amsfonts}%
\usepackage{amssymb}%
\usepackage{graphicx}
\usepackage{amsthm}
\usepackage{mathtools}
\usepackage{amssymb}
\usepackage{tikz}
\usepackage{tikz-cd}
\usetikzlibrary{matrix,arrows}
\usepackage[english, ngerman]{babel}
\usepackage[utf8]{inputenc}
\usepackage{hyperref}
\usepackage{cite}

\newtheorem{thm}{Theorem}
\theoremstyle{plain}
\newtheorem*{acknowledgement}{Acknowledgement}

\newtheorem{cor}{Corollary}

\newtheorem{defn}{Definition}

\newtheorem{lem}{Lemma}

\newtheorem{prop}{Proposition}
\newtheorem{rem}{Remark}

\numberwithin{equation}{section}

\renewcommand{\phi}{\varphi}

\newcommand{\BC}{{\mathbb{C}}}
\newcommand{\BD}{{\mathbb{D}}}

\newcommand{\BF}{{\mathbb{F}}}

\newcommand{\BH}{{\mathbb{H}}}

\newcommand{\BP}{{\mathbb{P}}}
\newcommand{\BQ}{{\mathbb{Q}}}
\newcommand{\BR}{{\mathbb{R}}}

\newcommand{\BZ}{{\mathbb{Z}}}

\newcommand{\CJ}{{\mathcal J}}

\newcommand{\OQ}{{\overline{\BQ}}}

\newcommand{\nin}{\notin}
\newcommand{\ord}{\mathop{\rm ord}\nolimits}
\renewcommand{\mod}{\mathop{\rm mod}\nolimits}

\newcommand{\nequiv}{\not \equiv}

\newcommand{\im}{\mathop{{\rm Im}}\nolimits}

\newcommand{\artanh}{\mathop{\rm artanh}\nolimits}

\renewcommand{\Re}{\mathop{\rm Re}\nolimits}
\renewcommand{\Im}{\mathop{\rm Im}\nolimits}

\newcommand{\SL}[1]{\mathop{\rm SL}_{#1} \nolimits}

\newcommand{\quotient}[2]{
        \mathchoice
            {
                \text{\raise1ex\hbox{$#1$}\Big/\lower1ex\hbox{$#2$}}%
            }
            {
                #1\,/\,#2
            }
            {
                #1\,/\,#2
            }
            {
                #1\,/\,#2
            }
    }
		
\newcommand{\rquotient}[2]{
        \mathchoice
            {
                \text{\lower1ex\hbox{$#1$}\Big \backslash \raise01ex\hbox{$#2$}}%
            }
            {
                #1\,\backslash\,#2
            }
            {
                #1\,\backslash\,#2
            }
            {
                #1\,\backslash\,#2
            }
    }
		
\newcommand{\lrquotient}[3]{
        \mathchoice
            {
                \text{\lower1ex\hbox{$#1$}\Big \backslash \raise01ex\hbox{$#2$}\Big/\lower1ex\hbox{$#3$}}%
            }
            {
                #1\,\backslash\,#2\,/\,#3
            }
            {
                #1\,\backslash\,#2\,/\,#3
            }
            {
                #1\,\backslash\,#2\,/\,#3
            }
    }

\begin{document}
\selectlanguage{english}

\bibliographystyle{alpha}

\title{Uniform bounds on sup-norms of holomorphic forms of real weight}
\author{Raphael S. Steiner}
\address{Department of Mathematics, University of Bristol, Bristol BS8 1TW, UK}%
\email{raphael.steiner@bristol.ac.uk}%




\begin{abstract} We establish uniform bounds for the sup-norms of modular forms of arbitrary real weight $k$ with respect to a finite index subgroup $\Gamma$ of $\SL2(\BZ)$. We also prove corresponding bounds for the supremum over a compact set. We achieve this by extending to a sum over an orthonormal basis $\sum_j y^k|f_j(z)|^2$ and analysing this sum by means of a Bergman kernel and the Fourier coefficients of Poincar\'e series. As such our results are valid without any assumption that the forms are Hecke eigenfunctions. Under some weak assumptions we further prove the right order of magnitude of $\sup_{z \in \BH} \sum_j y^k|f_j(z)|^2 $.
\end{abstract}
\maketitle

\section{Introduction}
Supremum norms of Maass and holomorphic cusp forms (of integral and half-integral weight) have been studied in various ways. Iwaniec-Sarnak \cite{IS95} obtained the first non-trivial result in the eigenvalue aspect. Since then, results have been obtained in the level and weight aspects by various authors. In particular, we refer to the following papers for the current best-known bounds for Hecke eigenforms: Xia \cite{Supnormintweight} for the weight aspect, Das–Sengupta \cite{DasSeng} for the weight aspect for forms on compact quotients, Harcos–Templier for the level aspect \cite{HT2},\cite{HT3}, Saha \cite{Saha14} for the level aspect in the non-squarefree case, Templier \cite{Thybrid} for a hybrid bound, and Kiral \cite{halflevel} for the level aspect in the case of half-integral weight. For non Hecke eigenforms, results in the weight aspect have been obtained by Rudnick \cite{R05} and Friedman-Jorgenson-Kramer \cite{FJK}. Recently, there has also been an explosion of papers related to the sup-norm question for automorphic forms of higher rank.


In this paper we study the supremum norm problem for holomorphic cusp forms in a variety of cases, where no non-trivial bounds have been previously written down. In particular we generalize the results obtained in \cite{FJK} in absolute uniformity to modular forms \emph{of arbitrary real weight} with respect to a finite index subgroup of $\SL2(\BZ)$. We note that the results of this paper do not assume that the cusp forms are eigenfunctions of any Hecke operators.\\

To motivate some of our results we recall a result of Rudnick \cite{R05}, who proved that for a fixed compact subset $K$ of the upper-half plane $\BH$ and a cusp form $f$ of weight $k \in 2 \BZ$ for the full modular group $\SL2(\BZ)$ we have $\sup_{z \in K} y^{\frac{k}{2}}|f(z)| \ll_K k^{\frac{1}{2}} \| f\|_2$. This result is essentially the best possible as there is a family of modular forms, which admit their supremum in a compact set and satisfy $\sup_{z \in K} y^{\frac{k}{2}}|f(z)| \gg k^{\frac{1}{2}- \epsilon} \| f\|_2$. We generalize this result of Rudnick uniformly to arbitrary real weight $k$, finite index subgroup $\Gamma$ and automorphy factor $\nu$ as follows:

\begin{thm} \label{thm:1} Let $\nu$ be an automorphy factor of weight $k\ge 6$ for a finite index subgroup $\Gamma \ni -I$ of $\SL2(\BZ)$ and $f\in S(\Gamma,k,\nu)$ with Petersson norm\footnote{The Petersson norm has been defined in this paper so as to be independent of $\Gamma$.} $\langle f,f \rangle_{\Gamma}=1$. Then we have for any compact subset $K \subseteq \BH$:
$$
\sup_{z \in K}y^{\frac{k}{2}}|f(z)| \ll_{K} [\SL2(\BZ):\Gamma]^{\frac{1}{2}} k^{\frac{1}{2}}.
$$
\end{thm}

The proof will essentially rely on the construction of a Bergman kernel for modular forms of real weight. We refer the reader to Theorem \ref{thm:bergman} for the key properties of the Bergman kernel.\\

If we do not restrict ourselves to compact sets the situation is different. In this case we are able to show:

\begin{thm} \label{thm:2} Let $\nu$ be an automorphy factor of weight $k\gg 1$ with respect to a finite index subgroup $\Gamma \ni -I$ of $\SL2(\BZ)$. Further let $f \in S(\Gamma,k,\nu)$ with $\langle f,f\rangle_{\Gamma}=1$ be a normalized cusp form. Then we have:
$$\begin{aligned}
\sup_{z \in \BH} y^{\frac{k}{2}}|f(z)| & \ll_{\epsilon} \left(1+ \max_{\tau \in \SL2(\BZ)}n_{\tau}^{\frac{1}{2}}k^{-\frac{1}{2}+\epsilon} \right) \frac{\displaystyle [\SL2(\BZ):\Gamma]^{\frac{1}{2}}k^{\frac{3}{4}}}{\displaystyle \min_{\tau \in \SL2(\BZ)}\eta_{\tau}^{\frac{1}{2}}}\\
& \ll_{\epsilon} \left(1+ [\SL2(\BZ):\Gamma]^{\frac{1}{2}}k^{-\frac{1}{2}+\epsilon} \right) \frac{\displaystyle [\SL2(\BZ):\Gamma]^{\frac{1}{2}}k^{\frac{3}{4}}}{\displaystyle \min_{\tau \in \SL2(\BZ)}\eta_{\tau}^{\frac{1}{2}}}.
\end{aligned}$$
\end{thm}

Here $n_{\tau}$ defines the width of the cusp $\tau \infty$. The quantity $\eta_{\tau}$ equals the cusp parameter\footnote{In the classical setting $\Gamma$ a congruence subgroup and automorphy factor $j^{k}$, where $k$ is an even integer, all cusp parameters are $0$.} of $\tau \infty$ if the cusp parameter is not equal to $0$, and equals $1$ if the cusp parameter equals $0$. We note that in the special case $\Gamma=\Gamma_0(N)$ the width $n_{\tau}$ is an integer dividing $N$.\\

Both Theorems \ref{thm:1} and \ref{thm:2} may be regarded as convexity results, since they are obtained by embedding f into an orthonormal basis and then proving the corresponding bounds for the sum over the whole basis in different regions. This generalizes the work of Friedman-Jorgenson-Kramer \cite{FJK} to real weight. One should remark here that our approach differs from theirs as we use properties of the Poincar\'e series and their work depends on the careful analysis of a heat kernel and what happens at the cusps.\\

If we assume $[\SL2(\BZ):\Gamma] \ll k^{1-\delta}$, we are able to show that our analysis of $\sum_j y^k |f_j(z)|^2$ gives the correct order of magnitude in the following sense:

\begin{thm} \label{thm:3}
Let $\nu$ be an automorphy factor of weight $k \gg 1$ for $\Gamma \ni -I$ a finite index subgroup of $\SL2(\BZ)$. Then we have for $[\SL2(\BZ):\Gamma] \ll k^{1-\delta} $ and $\{f_j\}$ an orthonormal basis of $S(\Gamma,k,\nu)$:
$$
\sup_{z \in \BH} \sum_{j} y^k |f_j(z)|^2 \asymp \frac{[\SL2(\BZ):\Gamma] k^{\frac{3}{2}}}{\displaystyle \min_{\tau \in \SL2(\BZ)}\eta_{\tau}},
$$
where the implied constant depends at most on $\delta$ and the implied constant in $[\SL2(\BZ):\Gamma] \ll k^{1-\eta}$.
\end{thm}

Finally, we remark that while the theory of Hecke operators of real weight has its difficulties, it has been well established in the case of half-integral weight. So, in analogy with the integral weight case, one expects that it should be possible to improve all the above bounds significantly in the case that $f$ is a Hecke eigenform of half-integral weight. This indeed turns out to be the case and will be topic of a forthcoming paper.
\newpage
\section{Notation and Preliminaries}
For $w \in \BC$ and $k \in \BR$ we let $w^k:=\exp(k \cdot \log(w))$, where $\log(w)=\log(|w|)+i \arg(w)$ with $-\pi < \arg(w) \le \pi$. The symbol $\ll$ denotes the Vinogradov symbol and $f(x) \ll_{A,B,C} g(x)$ means $|f(x)| \le K g(x)$, where $K$ depends at most on $A,B$ and $C$. Further the symbol $f(x) \asymp_{A} g(x)$ means that $f(x) \ll_{A} g(x)$ and $g(x) \ll_{A} f(x)$.\\

As usual the action of $\SL2(\BZ)$ on $\BH=\{z \in \BC | \Im z > 0\}$ is given by Möbius transformations:
$$
\gamma z=\gamma \cdot z = \frac{az+b}{cz+d}, \quad \forall \gamma= \begin{pmatrix} a & b \\ c & d \end{pmatrix} \in \SL2(\BZ), \forall z \in \BH.
$$
The action is then extended to the set of cusps $\OQ=\BP^1(\BQ)=\BQ \sqcup \{ \infty \}$. For $\gamma= \begin{pmatrix} a & b \\ c & d \end{pmatrix} \in \SL2(\BR)$ we define
$$
j(\gamma,z)=(cz+d), \forall z \in \BH.
$$
For a detailed treatment of modular forms of real weight we refer the reader to \cite{MFaF}. Here we recall the necessary facts. Throughout this paper we assume $-I \in \Gamma$ and $\Gamma$ is a finite index subgroup of $\SL2(\BZ)$. Also we denote $\hat{\Gamma}:=\quotient{\Gamma}{\{\pm I\}}$.
\begin{defn} A function $\nu: \Gamma \times \BH \to \BC$ is called an \emph{automorphy factor} of weight $k$ on $\Gamma$ if the following conditions are satisfied:
\begin{enumerate}
	\item $\forall \gamma \in \Gamma: \nu(\gamma,\cdot)$ is a holomorphic function on $\BH$,
	\item $\forall \gamma \in \Gamma, \forall z \in \BH: |\nu(\gamma,z)|=|j(\gamma,z)|^k$,
	\item $\forall \gamma,\tau \in \Gamma, \forall z \in \BH: \nu(\tau \gamma,z)=\nu(\tau,\gamma z)\nu(\gamma,z)$,
	\item $\forall \gamma \in \Gamma, z \in \BH: \nu(-\gamma,z)=\nu(\gamma,z)$.
\end{enumerate} \end{defn}

Corresponding to $\nu$ we can define a \emph{multiplier system} $\upsilon:\Gamma\to S^1$ of weight $k$ on $\Gamma$ as:
$$
\upsilon(\gamma)=\upsilon(\gamma,z) := \frac{\nu(\gamma,z)}{j(\gamma,z)^k}, \quad \forall \gamma \in \Gamma, \forall z \in \BH.
$$
We remark that the right hand side is indeed independent of $z$ as it is a bounded holomorphic function on $\BH$ and thus constant. It satisfies the relation:
$$
\upsilon(\tau\gamma)=\sigma(\tau,\gamma)\upsilon(\tau)\upsilon(\gamma), \quad \forall \tau,\gamma \in \Gamma,
$$
where
$$
\sigma(\tau,\gamma):=\frac{j(\tau,\gamma z)^kj(\gamma,z)^k}{j(\tau\gamma,z)^k}.
$$

If $\nu$ is an automorphy factor of weight $k$ on $\Gamma$ and $\tau \in \SL2(\BZ)$ we can define a conjugate automorphy factor $\nu^{\tau}$ of weight $k$ on $\Gamma^{\tau}:= \tau^{-1}\Gamma \tau$ in the following way:
$$
\nu^{\tau}(\tau^{-1}\gamma \tau,z)=\frac{\nu(\gamma,\tau z)j(\tau,z)^k}{j(\tau,\tau^{-1}\gamma\tau z)^k}, \quad \forall \gamma \in \Gamma, \forall z \in \BH.
$$

\begin{rem} If $\tau \in \Gamma$, then $\nu^{\tau}=\nu$, but $\Gamma^{\tau}=\Gamma$ does not necessarily imply $\nu^{\tau}=\nu$, if $\tau \nin \Gamma$. \end{rem}

\begin{defn} Let $\tau \in \SL2(\BZ)$. The \emph{width} $n_{\tau}$ of the cusp $\tau \infty$ is defined as the smallest natural number $n$ such that the stabilizer $\hat{\Gamma}_{\tau \infty}$ of $\tau \infty$ is generated by $\tau \begin{pmatrix}	1 & n \\ 0 & 1 \end{pmatrix}\tau^{-1}$. It is only dependent on the equivalence class of $\tau \infty$ modulo $\Gamma$.  \end{defn}

\begin{defn} Let $\tau \in \SL2(\BZ)$. The \emph{cusp parameter} $\kappa_{\tau}$ is defined as the real number in $[0,1)$, which satisfies:
$$
e^{2 \pi i \kappa_{\tau}}=\upsilon^{\tau}(U^{n_\tau})=\upsilon^{\tau}(U^{n_{\tau}})j(U^{n_{\tau}},z)^k=\nu^{\tau}(U^{n_{\tau}},z)=\upsilon(\tau U^{n_{\tau}} \tau^{-1}).
$$
It is only dependent on the equivalence class of $\tau \infty$ modulo $\Gamma$ and $\nu$.
\end{defn}

If $f$ is a meromorphic function on $\BH$, we define the $\gamma$-transform $f|_{k}\gamma$ of $f$ as:
$$
(f|_{k}\gamma)(z)=j(\gamma,z)^{-k}f(\gamma z).
$$

\begin{defn} A holomorphic function $f$ on $\BH$ is called a \emph{modular form} with respect to $\Gamma,k,\nu$ if it satisfies
$$
f(\gamma z)= \nu(\gamma,z)f(z)=\upsilon(\gamma) j(\gamma,z)^k f(z), \quad \forall \gamma \in \Gamma, z \in \BH
$$
and has a Fourier expansion at every cusp $\tau \infty$ of the form:
$$
(f|_k\tau)(z)=\sum_{\substack{m \in \BZ, \\ m+ \kappa_{\tau} \ge 0}} \widehat{(f|_k\tau)}(m) e^{\frac{2 \pi i (m+\kappa_{\tau}) z }{n_{\tau}}}.
$$
We say $f \in M(\Gamma,k,\nu)$. If moreover the sum can be restricted to $m+\kappa_{\tau}>0$ we say $f$ is a cusp form and we say $f \in S(\Gamma,k,\nu)$.
\end{defn}

\begin{rem} For $f \in S(\Gamma,k,\nu)$ it is evident, that $\eta_{\tau}=\min \{r | r \in \BR^+ \text{ and } r-\kappa_{\tau} \in \BZ\} $ parametrises a lower bound on the exponential decay at the cusp $\tau \infty$.
\end{rem}

The spaces $M(\Gamma,k,\nu)$ and $M(\Gamma,k,\nu)$ are finite dimensional and the latter can be made into a Hilbert space, by defining an inner product.

\begin{defn} The Petersson inner product on $S(\Gamma,k,\nu)$ is defined by
$$
\langle f,g \rangle_{\Gamma} = \frac{1}{\mu(\Gamma)} \int_{\BF_{\Gamma}} f(z)\overline{g(z)}y^k \frac{dxdy}{y^2},
$$
where $\mu(\Gamma)=[\SL2(\BZ):\Gamma]=[\operatorname{PSL}_2(\BZ):\hat{\Gamma}]$. It is indeed an inner product and is independent of the choice of the fundamental domain $\BF_{\Gamma}$ and independent of the subgroup $\Gamma$, i.e. if $\Gamma' \le \Gamma$ of finite index, then
$$
\langle f,g \rangle_{\Gamma}=\langle f,g \rangle_{\Gamma'}, \quad \forall f,g \in S(\Gamma,k,\nu) \supseteq S(\Gamma',k,\nu).
$$
\end{defn}

\begin{defn} For $k>2$ we define the $m$-th Poincar\'e series of weight $k$ at the cusp $\tau^{-1}\infty$, where $\tau \in \SL2(\BZ)$, with respect to $\Gamma, \nu$ as:
$$
G_{\tau}(\Gamma,k,\nu;z,m)=\sum_{\gamma \in \rquotient{\hat{\Gamma}_{\tau^{-1}\infty}}{\hat{\Gamma}}} \frac{\exp \left( \displaystyle \frac{2\pi i (m + \kappa_{\tau^{-1}})}{n_{\tau^{-1}}} \tau \gamma z \right)}{j(\tau,\gamma z)^k\nu(\gamma,z)}.
$$
\end{defn}

\begin{prop} The above Poincar\'e series converges locally uniformly on $\BH$ and defines thus a holomorphic function on $\BH$ and defines an unrestricted modular form of weight $k$ with respect to $\Gamma, \nu$.
\begin{enumerate}
	\item They satisfy the relations:
	$$
	G_{\tau_1}(\Gamma,k,\nu; \cdot,m)|_k \tau_2= \frac{1}{\sigma(\tau_1,\tau_2)}G_{\tau_1 \tau_2}(\Gamma^{\tau_2},k,\nu^{\tau_2}; \cdot,m);
	$$
	\item If $m+\kappa_{\tau^{-1}}>0$, then $G_{\tau}(\Gamma,k,\nu;\cdot,m) \in S(\Gamma,k,\nu)$;
	\item If $m+\kappa_{\tau^{-1}}=0$, then $G_{\tau}(\Gamma,k,\nu;\cdot,m) \in M(\Gamma,k,\nu)$ non-zero with $\ord(G_{\tau},\tau^{-1}\infty,\Gamma)=0$ and at the other cusps $\zeta \nequiv \tau^{-1}\infty \mod \Gamma$ one has $\ord(G_{\tau},\zeta,\Gamma)>0$;
	\item If $m+\kappa_{\tau^{-1}}<0$, then $\ord(G_{\tau},\tau^{-1}\infty,\Gamma)=m+\kappa$ and at the other cusps $\zeta \nequiv \tau^{-1}\infty \mod \Gamma$ one has $\ord(G_{\tau},\zeta,\Gamma)>0$.
\end{enumerate}
\end{prop}

\begin{proof} See \cite{MFaF} Theorem 5.1.2 page 136.
\end{proof}

\begin{prop} We have for $k>2$, $f \in S(\Gamma,k,\nu)$, $\tau \in \SL2(\BZ) \backslash \{-U^l |l \in \BZ \}$, $m+\kappa_{\tau} \ge 0$:
$$
\langle f, G_{\tau}(\Gamma,k,\nu; \cdot ,m)\rangle = \frac{n_{\tau^{-1}}^k\Gamma(k-1)}{\mu(\Gamma) (4 \pi (m+\kappa_{\tau^{-1}}))^{k-1}}  \widehat{(f|_k \tau^{-1})}(m).
$$
\end{prop}

\begin{proof} See \cite{MFaF} Theorem 5.2.2 page 149.
\end{proof}

\begin{prop} For $\tau \in \SL2(\BZ)$ the Poincar\'e series with $k>2$ satisfy the following equality:
$$
G_{\tau}(\Gamma,k,\nu; z,m)=\delta_{\tau} e^{\frac{2\pi i(m+\kappa_{I})z}{n_{I}}} + \sum_{r+\kappa_I>0} a(r,m;\tau) e^{\frac{2 \pi i (r+\kappa_{\tau^{-1}})z}{n_{\tau^{-1}}}},
$$
where 
$$
\delta_{\tau}= \begin{cases}\frac{\displaystyle e^{\frac{2\pi i s(m+\kappa_{I})}{n_I}}}{\upsilon(\tau^{-1}U^s)\sigma(\tau,\tau^{-1})}, & \text{if } \tau^{-1}U^s \in \Gamma, \text{for some }s \in \BZ, \\ 0, & \text{else}, \end{cases}
$$
and
$$
a(r,m;\tau)=\begin{cases} \displaystyle
\frac{(2 \pi)^k}{\Gamma(k)}i^{-k}(r+\kappa_I)^{k-1}\sum_{c=1}^{\infty} \frac{W(\Gamma,\nu;r,m;c)}{(n_Ic)^k}, & \text{if } m+\kappa_{\tau^{-1}}=0, \\ \displaystyle

2\pi i^{-k} \left(\frac{n_{\tau^{-1}}}{n_I}\right)^{\frac{k-1}{2}} \left( \frac{r+\kappa_I}{m+\kappa_{\tau^{-1}}} \right)^{\frac{k-1}{2}} \\ \displaystyle \quad \times \sum_{c=1}^{\infty} \frac{W(\Gamma,\nu;r,m;c)}{n_I c}J_{k-1}\left( \frac{4\pi}{c} \sqrt{\frac{(r+\kappa_I)(m+\kappa_{\tau^{-1}})}{n_In_{\tau^{-1}}}} \right), & \text{if }m+\kappa_{\tau^{-1}}>0, \\ \displaystyle

2\pi i^{-k} \left(\frac{n_{\tau^{-1}}}{n_I}\right)^{\frac{k-1}{2}} \left | \frac{r+\kappa_I}{m+\kappa_{\tau^{-1}}} \right |^{\frac{k-1}{2}} \\ \displaystyle \quad \times \sum_{c=1}^{\infty} \frac{W(\Gamma,\nu;r,m;c)}{n_I c}I_{k-1}\left( \frac{4\pi}{c} \sqrt{\frac{(r+\kappa_I)|m+\kappa_{\tau^{-1}}|}{n_In_{\tau^{-1}}}} \right), & \text{if }m+\kappa_{\tau^{-1}}<0.
\end{cases}
$$
The Bessel functions $J_{k-1},I_{k-1}$ are given by:
$$\begin{aligned}
J_{k-1}(z) &= \sum_{m=0}^{\infty} \frac{(-1)^m(\frac{z}{2})^{2m+k-1}}{\Gamma(m+1)\Gamma(m+k)}, \\
I_{k-1}(z) &= \sum_{m=0}^{\infty} \frac{(\frac{z}{2})^{2m+k-1}}{\Gamma(m+1)\Gamma(m+k)},
\end{aligned}$$
and the generalized Kloosterman sum is given by:
$$
W(\Gamma,\nu;r,m;c)= \sum_{ \displaystyle \gamma = \begin{pmatrix} a & * \\ c & d \end{pmatrix} \in \tau \CJ_{\tau} }
\frac{\displaystyle \exp\left( \frac{2\pi i}{c} \left( \frac{(m+\kappa_{\tau^{-1}})a}{n_{\tau^{-1}}}+ \frac{(r+\kappa_I)d}{n_I} \right) \right)}{\upsilon(\tau^{-1}\gamma)\sigma(\tau,\tau^{-1})}\sigma(\tau^{-1},\gamma),
$$
where $\CJ_{\tau}$ is the double coset
$$
\CJ_{\tau} = \lrquotient{\hat{\Gamma}_{\tau^{-1}\infty}}{\hat{\Gamma} - \tau^{-1}\{U^s|s \in \BZ\}}{\hat{\Gamma}_{\infty}}.
$$
\end{prop}

\begin{rem} If $\delta_\tau \neq 0$, then $n_I=n_{\tau^{-1}}$ and $\kappa_I=\kappa_{\tau^{-1}}$.
\end{rem}

\begin{proof} See \cite{MFaF} Theorem 5.3.2 page 162.
\end{proof}

\begin{cor} \label{cor:Fouriersqr} Let $\{f_j\}$ be an orthonormal basis of $S(\Gamma,k,\nu)$, $\tau \in \SL2(\BZ) \backslash \{-U^l |l \in \BZ \}$ and $m+\kappa_{\tau^{-1}}\ge 0$, then we have:
$$\begin{aligned}
\sum_j |\widehat{(f_j|_k\tau^{-1})}(m)|^2 = &\frac{\mu(\Gamma)(4\pi(m+\kappa_{\tau^{-1}}))^{k-1}}{n_{\tau^{-1}}^k\Gamma(k-1)} \\ &\quad \times \left(1+2\pi i^{-k} \sum_{c=1}^{\infty} \frac{W(\Gamma^{\tau^{-1}},\nu^{\tau^{-1}};m,m;c)}{n_{\tau^{-1}}c} J_{k-1}\left( \frac{4\pi(m+\kappa_{\tau^{-1}})}{cn_{\tau^{-1}}} \right) \right).
\end{aligned}$$
\end{cor}
\begin{proof}
We have
$$\begin{aligned}
G_{\tau}(\Gamma,k,\nu;z,m) &= \sum_j \langle G_{\tau}(\Gamma,k,\nu;\cdot,m) ,f_j \rangle f_j(z) \\ &= \frac{n_{\tau^{-1}}^k\Gamma(k-1)}{\mu(\Gamma) (4 \pi (m+\kappa_{\tau^{-1}}))^{k-1}} \sum_j \overline{\widehat{(f_j|_k \tau^{-1})}(m)}f_j(z)
\end{aligned}$$
and henceforth
$$\begin{aligned}
\sum_j\overline{\widehat{(f_j|_k \tau^{-1})}(m)}(f_j|_k\tau^{-1})(z) &= \frac{\mu(\Gamma)(4\pi(m+\kappa_{\tau^{-1}}))^{k-1}}{n_{\tau^{-1}}^k\Gamma(k-1)} (G_{\tau}(\Gamma,k,\nu;\cdot,m)|_k \tau^{-1})(z) \\
&= \frac{\mu(\Gamma)(4\pi(m+\kappa_{\tau^{-1}}))^{k-1}}{n_{\tau^{-1}}^k\Gamma(k-1)} G_{I}(\Gamma^{\tau^{-1}},k,\nu^{\tau^{-1}};z,m).
\end{aligned}$$
Using the previous theorem we can easily deduce the $m$-th Fourier coefficient at $\infty$. To verify the equality one easily checks that $\delta_I=1$ and $n_I$ for $\Gamma^{\tau^{-1}}$ is equal to $n_{\tau^{-1}}$ for $\Gamma$. And we refer to \cite{MFaF} to see that $\kappa_I$ for $\Gamma^{\tau^{-1}},\nu^{\tau^{-1}}$ is equal to $\kappa_{\tau^{-1}}$ for $\Gamma,\nu$.
\end{proof}

\begin{defn} We define the Bergman kernel for $\Gamma, \nu, k>2$ on $\BH^2$ as
$$
h(z,w)=\sum_{\gamma \in \Gamma} \frac{1}{\left(\frac{w+\gamma z}{2i}\right)^k \nu(\gamma,z)}.
$$ \end{defn}

\begin{thm}\label{thm:bergman} The Bergman kernel satisfies the following properties:
\begin{enumerate}
	\item The sum converges absolutely uniformly on the sets $\{z \in \BH | \epsilon < \arg(z) < \pi - \epsilon \} \times \{w \in \BH | \im w > \epsilon\}$,
	\item $\forall w \in \BH: h(\cdot,w) \in S(\Gamma,k,\nu)$,
	\item $\forall f \in S(\Gamma,k,\nu)$: $$ \langle f, h(\cdot,\overline{-w})\rangle = \frac{2}{\mu(\Gamma)} \cdot \frac{4\pi}{k-1} \cdot f(w). $$
\end{enumerate}
\end{thm}

\begin{proof} We have for $\tau \in \SL2(\BZ)$:
$$\begin{aligned}
|(h(\cdot,w)|_k \tau)(z)| &=\left|\frac{1}{j(\tau,z)^k} \sum_{\gamma \in \Gamma} \frac{1}{\left(\frac{w+\gamma \tau z}{2i}\right)^k \nu(\gamma, \tau z)}  \right| \\
&\le \sum_{\gamma \in \Gamma} \frac{1}{\left|\frac{w+\gamma \tau z}{2i}\right|^k |j(\gamma, \tau z)|^k |j(\tau,z)|^k} \\
&= \sum_{\gamma \in \Gamma} \frac{1}{\left|\frac{w+\gamma \tau z}{2i}\right|^k |j(\gamma \tau, z)|^k} \\
&\le \sum_{\gamma \in \SL2(\BZ)} \left|\frac{\im w}{2}\right|^{-k} \frac{1}{|j(\gamma,z)|^k}.
\end{aligned}$$
The latter converges uniformly on the mentioned sets. It follows, that $h(\cdot,w)$ is a holomorphic function on the upper-half plane. Moreover we can exchange limit and summation in the following:
$$\begin{aligned}
\lim_{\im z \to \infty} |(h(\cdot,w)|_k \tau)(z)| &\le \lim_{\im z \to \infty} \sum_{\gamma \in \SL2(\BZ)} \frac{1}{\left|\frac{w+\gamma z}{2i}\right|^k |j(\gamma, z)|^k} \\
&= \sum_{\gamma \in \SL2(\BZ)} \lim_{\im z \to \infty} \frac{1}{\left|\frac{w+\gamma z}{2i}\right|^k |j(\gamma, z)|^k} =0.
\end{aligned}$$
To see the latter we distinguish two cases. Let $\gamma = \begin{pmatrix} a & b\\ c & d \end{pmatrix}$. If $c=0$ then
$$
\lim_{\im z \to \infty} |w+\gamma z| = \lim_{\im z \to \infty} |w+d^{-1}(az+b)| = \infty
$$
and $j(\gamma,z)=d$. If $c \neq 0$, then $|w+\gamma z | \ge \im w$ and
$$
\lim_{\im z \to \infty} |j(\gamma,z)| = \lim_{\im z \to \infty} |cz+d|= \infty.
$$
We also have for $\tau \in \Gamma$:
$$
h(\tau z, w)= \sum_{\gamma \in \Gamma} \frac{1}{\left(\frac{w+\gamma \tau z}{2i}\right)^k \nu(\gamma, \tau z)} = \sum_{\gamma \in \Gamma} \frac{\nu(\tau,z)}{\left(\frac{w+\gamma \tau z}{2i}\right)^k \nu(\gamma \tau, z)} = \nu(\tau,z) h(z,w).
$$
From which the second claim follows. The third claim needs more work. Let $f \in S(\Gamma,k,\nu)$. We have:
$$\begin{aligned}
f(z)\overline{h(z,\overline{-w})}y^k &= \sum_{\gamma \in \Gamma} \frac{f(z)y^k}{\left(\frac{w- \overline{\gamma z}}{2i}\right)^k \overline{\nu(\gamma,z)}} = \sum_{\gamma \in \Gamma} \frac{f(\gamma z)y^k}{\left(\frac{w- \overline{\gamma z}}{2i}\right)^k \overline{\nu(\gamma,z)} \nu(\gamma, z)} \\
&= \sum_{\gamma \in \Gamma} \frac{f(\gamma z)(\im \gamma z)^k}{\left(\frac{w- \overline{\gamma z}}{2i}\right)^k}.
\end{aligned}$$
Plugging this in the definition of the Petersson inner product we find:
$$\begin{aligned}
\langle f, h(\cdot,\overline{-w}) \rangle &= \frac{1}{\mu(\Gamma)} \int_{\BF_{\Gamma}} \sum_{\gamma \in \Gamma} \frac{f(\gamma z)(\im \gamma z)^k}{\left(\frac{w- \overline{\gamma z}}{2i}\right)^k} \frac{dxdy}{y^2} \\
&= \frac{1}{\mu(\Gamma)} \sum_{\gamma \in \Gamma} \int_{\gamma\BF_{\Gamma}}  \frac{f(z)y^k}{\left(\frac{w- \overline{z}}{2i}\right)^k} \frac{dxdy}{y^2} \\
&= \frac{2}{\mu(\Gamma)} \int_{\BH}  \frac{f(z)y^k}{\left(\frac{w- \overline{z}}{2i}\right)^k} \frac{dxdy}{y^2}.
\end{aligned}$$
Using a Cayley transformation $l_w: \BH \to \BD, z \mapsto \zeta = (z-w)/(z-\overline{w})$, which maps the upper-half plane biholomorphic to the unit disk, we will transform the integral. For this we denote $z=x+iy, w=u+iv, \zeta=\xi+i\eta$. We have the following identities:
$$
\frac{dl_w}{dz}=\frac{w-\overline{w}}{(z-\overline{w})^2}=\frac{2v}{(z-\overline{w})^2} \Rightarrow d\xi d\eta= \frac{4v^2}{|z-\overline{w}|^4}dxdy,
$$
$$
1-|\zeta|^2=\frac{|z-\overline{w}|^2-|z-w|^2}{|z-\overline{w}|^2}=\frac{4yv}{|z-\overline{w}|^2}.
$$
Back to the integral we have to calculate:
$$\begin{aligned}
\int_{\BH}  \frac{f(z)y^k}{\left(\frac{w- \overline{z}}{2i}\right)^k} \frac{dxdy}{y^2} &= \int_{\BH}f(z) \left(\frac{w-\overline{z}}{2i}\right)^{-k} \left[\frac{(1-|\zeta|^2)|z-\overline{w}|^2}{4v} \right]^k \left[\frac{4v}{(1-|\zeta|^2)|z-\overline{w}|^2}\right]^2dxdy \\
&= \int_{\BD}f(z) \left(\frac{w-\overline{z}}{2i}\right)^{-k}(1-|\zeta|^2)^k \left|\frac{w-\overline{z}}{2i} \right|^{2k}v^{-k} \frac{4d\xi d\eta}{(1-|\zeta|^2)^2} \\
&= \frac{4}{v^k} \int_{\BD} f(z) \left(\frac{z-\overline{w}}{2i}\right)^k(1-|\zeta|^2)^k \frac{d\xi d\eta}{(1-|\zeta|^2)^2} \\
&= \frac{4}{v^k} \int_{\BD} f^{\dagger}(\zeta)(1-|\zeta|^2)^k \frac{d\xi d\eta}{(1-|\zeta|^2)^2}.
\end{aligned}$$
Where
$$
f^{\dagger}(\zeta)=f(z)\left(\frac{z-\overline{w}}{2i}\right)^k
$$
is holomorphic on $\BD$ and satisfies
$$
\bigl |f^{\dagger}(\zeta)(1-|\zeta|^2)^{\frac{k}{2}}\bigr|= v^{\frac{k}{2}} y^{\frac{k}{2}}|f(z)| \ll_{f,w,k} 1.
$$
By computing the following integral for $0 \le t<1,\alpha>-1$
$$\begin{aligned}
\int_{t\BD} (1-|\zeta|^2)^{\alpha}d\xi d \eta &= \int_0^t \int_0^{2 \pi} (1-r^2)^\alpha r d\phi dr \\
&= -\pi\frac{(1-r^2)^{\alpha+1}}{\alpha+1} \biggr |_{r=0}^t\\
&= \frac{\pi}{\alpha+1}\left(1-(1-t^2)^{\alpha+1}\right)
\end{aligned}$$
we see that the integral left to be computed converges absolutely uniformly for $k>2$. Hence we have:
$$\begin{aligned}
\int_{\BD} f^{\dagger}(\zeta)(1-|\zeta|^2)^k \frac{d\xi d\eta}{(1-|\zeta|^2)^2} &= \lim_{t \to 1^-} \int_{t\BD} f^{\dagger}(\zeta)(1-|\zeta|^2)^k \frac{d\xi d\eta}{(1-|\zeta|^2)^2} \\
&= \lim_{t \to 1^-} \int_{t\BD} \sum_{n=0}^{\infty} \left(f^{\dagger}\right)^{(n)}(0)\frac{\zeta^n}{n!}(1-|\zeta|^2)^k \frac{d\xi d\eta}{(1-|\zeta|^2)^2} \\
&= \lim_{t \to 1^-} \sum_{n=0}^{\infty} \frac{\left(f^{\dagger}\right)^{(n)}(0)}{n!} \int_{t\BD} \zeta^n(1-|\zeta|^2)^k \frac{d\xi d\eta}{(1-|\zeta|^2)^2}.
\end{aligned}$$
As the Taylor expansion converges absolutely uniformly on $t\BD$. Making the substitution $\zeta \mapsto e^{2\pi i s}\zeta$ for some suitable $s\in \BR$ we see that:
$$
\int_{t\BD} \zeta^n(1-|\zeta|^2)^k \frac{d\xi d\eta}{(1-|\zeta|^2)^2} = \begin{cases} \frac{\pi}{k-1}\left(1-(1-t^2)^{k-1}\right), & \text{if } n=0, \\ 0, & \text{if } n>0. \end{cases}
$$
And therefore we have:
$$\begin{aligned}
\int_{\BD} f^{\dagger}(\zeta)(1-|\zeta|^2)^k \frac{d\xi d\eta}{(1-|\zeta|^2)^2} &= f^{\dagger}(0)\cdot \frac{\pi}{k-1} = \frac{\pi}{k-1}f(w)v^k,
\end{aligned}$$
which completes the proof.
\end{proof}

\begin{cor} \label{cor:Bergmanavg} Let $\{f_j\}$ be an orthonormal basis of $S(\Gamma,k,\nu)$ and $\tau \in \SL2(\BZ)$. Then we have:
$$
\sum_{j}|(f_j|_k \tau)(z)|^2 = \mu(\Gamma) \frac{k-1}{8\pi} \sum_{\gamma \in \Gamma^{\tau}} \frac{1}{\left(\frac{\gamma z - \overline{z}}{2i}\right)^k\nu^{\tau}(\gamma,z)}.
$$
\end{cor}
\begin{proof} We prove first the case $\tau=I$, which follows easily from the formula:
$$
h(z,\overline{-w})=\sum_j \langle h(\cdot,\overline{-w} ,f_j \rangle f_j(z) = \frac{2}{\mu(\Gamma)}\cdot \frac{4\pi}{k-1}\sum_j \overline{f_j(w)}f_j(z).
$$
For the general case we use the special case and the fact that $\{f_j|_k \tau\}$ is a basis of $S(\Gamma^{\tau},k,\nu^{\tau})$. The orthonormality follows from:
$$\begin{aligned}
\langle f,g \rangle_{\Gamma} &= \frac{1}{\mu(\Gamma)} \int_{\BF_{\Gamma}} f(z)\overline{g(z)}y^k \frac{dxdy}{y^2} \\
&= \frac{1}{\mu(\Gamma)} \int_{\tau^{-1}\BF_{\Gamma}} (f|_k\tau)(z) \overline{(g|_k\tau)(z)}y^k \frac{dxdy}{y^2} \\
&= \langle f|_k\tau, g|_k\tau \rangle_{\Gamma^{\tau}},
\end{aligned}$$
where we used $\mu(\Gamma)=\mu(\Gamma^{\tau})$, the $\SL2(\BZ)$ invariance of the measure $y^{-2}dxdy$ and that $\tau^{-1}\BF_{\Gamma}$ is a fundamental domain for $\Gamma^{\tau}$.
\end{proof}

We will also need uniform results on Bessel functions. Recall that the (modified) Bessel functions are given by
$$\begin{aligned}
J_{\rho}(x) &= \sum_{m=0} \frac{(-1)^m\left(\frac{x}{2}\right)^{2m+\rho}}{\Gamma(m+1)\Gamma(m+\rho+1)}, \\
Y_{\rho}(x) &= \sin(\rho \pi)^{-1} \left[ J_{\rho}(x)\cos(\rho \pi)-J_{-\rho}(x) \right], \\
I_{\rho}(x) &= \sum_{m=0} \frac{\left(\frac{x}{2}\right)^{2m+\rho}}{\Gamma(m+1)\Gamma(m+\rho+1)}, \\
K_{\rho}(x) &= \frac{\pi}{2} \sin(\rho \pi)^{-1} \left[ I_{-\rho}(x)-I_{\rho}(x) \right].
\end{aligned}$$

\begin{prop}\label{thm:Besselslarge} On has for $x \ge C>0$:
$$\begin{aligned}
|J_{\rho}(x)| & \ll_{C,\rho} x^{-\frac{1}{2}}, \\
|Y_{\rho}(x)| & \ll_{C,\rho} x^{-\frac{1}{2}}, \\
|K_{\rho}(x)| & \ll_{C,\rho} x^{-\frac{1}{2}} e^{-x}.
\end{aligned}$$
\end{prop}

\begin{proof} See \cite{ToBF} page 199, 202.
\end{proof}

\begin{prop}[Langer's formulas]\label{thm:Langer}
The Bessel function admits the following uniform formula for $x>\rho$:
$$
J_{\rho}(x)=w^{-\frac{1}{2}}(w-\arctan w)^{\frac{1}{2}} \left[ \frac{\sqrt{3}}{2}J_{\frac{1}{3}}(z)-\frac{1}{2}Y_{\frac{1}{3}}(z) \right] + O(\rho^{-\frac{4}{3}}),
$$
where
$$
w=\sqrt{\frac{x^2}{\rho^2}-1} \text{ and } z=\rho(w-\arctan(w)).
$$
For $x<\rho$ one has the formula
$$
J_{\rho}(x)=\pi^{-1}w^{-\frac{1}{2}}(\artanh(w)-w)^{\frac{1}{2}}K_{\frac{1}{3}}(z) + O(\rho^{-\frac{4}{3}}),
$$
where
$$
w=\sqrt{1-\frac{x^2}{\rho^2}} \text{ and } z= \rho(\artanh(w)-w).
$$
\end{prop}

\begin{proof} See \cite{HTF} page 30, 89.
\end{proof}

\begin{prop}\label{thm:JBesselmedasym}
For the intermediate range $|x-\rho| = o( \rho^{\frac{1}{3}})$ we have the following asymptotic for every $M$:
$$
J_{\rho}(x) = \frac{1}{3\pi} \sum_{m=0}^{M-1} B_m(x-\rho) \sin \left( \frac{\pi}{3}(m+1) \right) \frac{\Gamma \left(\frac{1}{3}(m+1)\right)}{\left( \frac{x}{6} \right)^{\frac{1}{3}(m+1)}} + O \left(x^{-\frac{M+1}{3}} \right).
$$

\end{prop}

\begin{proof} See \cite{ToBF} page 245-247.
\end{proof}

The proof given there is also enough to show the following proposition.

\begin{prop} \label{prop:JBesselmed} For $|x-\rho| \le C \rho^{\frac{1}{3}}, \rho \gg_C 1$ we have the following:
$$
|J_{\rho}(x)| \ll_{C} \rho^{-\frac{1}{3}}.
$$
\end{prop}


\begin{prop}\label{prop:JBesselverysmall}
One has for $\rho \ge 2 x^2$:
$$
|J_{\rho}(x)| \ll \frac{\left(\frac{x}{2} \right)^{\rho}}{\Gamma(\rho+1)}.
$$
\end{prop}

\begin{proof} Using Stirling approximation for the $\Gamma$-function one checks that:
$$\begin{aligned}
|J_{\rho}(x) |&= \left|\sum_{m=0} \frac{(-1)^m\left(\frac{x}{2}\right)^{2m+\rho}}{\Gamma(m+1)\Gamma(m+\rho+1)} \right | \\
& \ll \frac{\left(\frac{x}{2} \right)^{\rho}}{\Gamma(\rho+1)} \sum_{m=0}^{\infty} \frac{(\rho+1)^{\rho+\frac{1}{2}}}{(m+1)^{m+\frac{1}{2}}(\rho+m+1)^{\rho+m+\frac{1}{2}}e^{-2m}} \left( \frac{x}{2} \right)^{2m} \\
& \ll \frac{\left(\frac{x}{2} \right)^{\rho}}{\Gamma(\rho+1)} \sum_{m=0}^{\infty} \left( \frac{\rho+1}{\rho+m+1} \right)^{\rho} \left( \frac{x^2 e^2}{4(m+1)(\rho+m+1)} \right)^m \\
& \ll \frac{\left(\frac{x}{2} \right)^{\rho}}{\Gamma(\rho+1)} \sum_{m=0}^{\infty} \left( \frac{x^2 e^2}{4(\rho+1)} \right)^m \\
& \ll \frac{\left(\frac{x}{2} \right)^{\rho}}{\Gamma(\rho+1)}.
\end{aligned}$$
\end{proof}

\begin{prop} \label{prop:JBesselsmall} There exists $C'>0$ for which we have in the range $x \le \rho- C' \rho^{\frac{1}{3}}(\log \rho)^{\frac{1}{3}}, \rho \gg_{C'} 1$ the following estimation:
	$$
	|J_{\rho}(x)| \ll_{C'} \rho^{-\frac{4}{3}}.
	$$
\end{prop}

\begin{proof} We are going to use Langer's formula (see Proposition \ref{thm:Langer}) for $x<\rho$. There $z=\rho \sum_{n=1}^{\infty} \frac{w^{2n+1}}{2n+1} \ge \log \rho$ for a particular choice of $C'$ and we estimate by using Proposition \ref{thm:Besselslarge}:
$$\begin{aligned}
|J_{\rho}(x)| &= |\pi^{-1} w^{-\frac{1}{2}} (\artanh(w)-w)^{\frac{1}{2}}K_{\frac{1}{3}}(z)| + O(\rho^{-\frac{4}{3}}) \\
& \ll_{C'} (\rho w)^{-\frac{1}{2}} e^{-z} + O(\rho^{-\frac{4}{3}}) \\
& \ll_{C'} \left(2 C' \rho^{\frac{4}{3}}(\log \rho)^{\frac{1}{3}}-C'^2 \rho^{\frac{2}{3}} (\log \rho)^{\frac{2}{3}}  \right)^{-\frac{1}{4}} \rho^{-1} + O(\rho^{-\frac{4}{3}}) \\
& \ll_{C'} \rho^{-\frac{4}{3}}.
\end{aligned}$$

\end{proof}

A similar argument can be applied to get the next proposition.

\begin{prop} \label{prop:JBesselcrapgapsmall} We have for the range $ \rho - C\rho^{\frac{1}{3}}  \ge x \ge \rho- C \rho^{\frac{1}{3}}(\log \rho)^{\frac{1}{3}}, \rho \gg_C 1$ the following estimation:
$$
|J_{\rho}(x)| \ll_{C} \rho^{-\frac{1}{3}}.
$$
\end{prop}

\begin{prop} \label{prop:JBessellarge} For $x \ge \rho + C \rho^{\alpha}, \rho \gg_{C,\alpha} 1$ we have:
$$
|J_{\rho}(x)| \ll_C \begin{cases} 
\rho^{-\frac{\alpha+1}{4}}, & \text{for } \frac{1}{3} \le \alpha \le 1, \\
x^{-\frac{1}{2}} \ll_C \rho^{-\frac{\alpha}{2}}, & \text{for } 1 \le \alpha \le \frac{8}{3}, \\
\rho^{-\frac{4}{3}}, & \text{for } \frac{8}{3} \le \alpha.
\end{cases}
$$
\end{prop}

\begin{proof} Use Langer's formula (see Proposition \ref{thm:Langer}) for $x>\rho$. And note that for $\alpha \ge \frac{1}{3}$ and $\rho \gg_{C,\alpha} 1$ we have $z \gg 1$. We can thus use Proposition \ref{thm:Besselslarge} to deduce:
$$\begin{aligned}
|J_{\rho}(x)| &= \left| w^{-\frac{1}{2}}(w-\arctan w)^{\frac{1}{2}} \left[ \frac{\sqrt{3}}{2}J_{\frac{1}{3}}(z)-\frac{1}{2}Y_{\frac{1}{3}}(z) \right] \right|  + O(\rho^{-\frac{4}{3}}) \\
& \ll_C (\rho w)^{-\frac{1}{2}} + O(\rho^{-\frac{4}{3}}) \\
& \ll_C (x^2-\rho^2)^{-\frac{1}{4}} + O(\rho^{-\frac{4}{3}}).
\end{aligned}$$
From which the proposition follows.
\end{proof}
\newpage
\section{Proofs of Theorems}
The proofs of Theorem \ref{thm:1} and \ref{thm:2} will be based on the following simple inequality
\begin{equation}
\sup_{z \in \BH} y^{\frac{k}{2}}|f(z)| \le \max_{\tau \in \rquotient{\Gamma}{\SL2(\BZ)}} \sup_{z \in \BF_{I}} \sqrt{y^k\sum_j |(f_j|_k \tau)(z)|^2 },
\label{eq:convexbound}
\end{equation}
where $\BF_I$ is the standard fundamental domain for $\SL2(\BZ)$ and $\{f=f_1,f_2,\dots\}$ is a (finite) orthonormal basis. This inequality is easily seen to be true as $y^{\frac{k}{2}}|f(z)|$ is $\Gamma$ invariant and $\im(\tau z)^{\frac{k}{2}}|f(\tau z)|=y^{\frac{k}{2}}|(f|_k \tau)(z)|$.\\

For the proof of Theorem \ref{thm:2} we will use Corollary \ref{cor:Bergmanavg}. For this purpose it is sufficient to bound the following sum for $z \in \BF_{I}$:
\begin{equation}
\sum_{\gamma \in \Gamma^{\tau}} \frac{y^k}{\left(\frac{\gamma z - \overline{z}}{2i}\right)^k\nu^{\tau}(\gamma,z)} \le \sum_{\gamma \in \SL2(\BZ)} \frac{y^k}{\left|\frac{\gamma z - \overline{z}}{2i}\right|^k|j(\gamma,z)|^k}= \sum_{\gamma \in \SL2(\BZ)} \frac{(yy')^{\frac{k}{2}}}{\left( \left( \frac{x-x'}{2} \right)^2 + \left( \frac{y+y'}{2} \right)^2  \right)^{\frac{k}{2}}},
\label{eq:method2}
\end{equation}
where $x'+iy'=z'=\gamma z$. Our first estimation will be crude and will be used to deal with the cases $y\ll 1$ and $k= 3$, which then allows a good treatment of the sum, when $y$ is relatively small in comparison to $k$. First notice, that by AM-GM we have
$$
\left(\frac{y+y'}{2} \right)^2 \ge yy'.
$$
Thus every term is $\le 1$. For $c>0, (c,d)=1$ fix a matrix $\gamma_{c,d}= \begin{pmatrix} * & * \\ c & d \end{pmatrix}$ with $|\Re(\gamma_{c,d}z-z)|\le \frac{1}{2}$ then we have:
$$\begin{aligned}
\sum_{\gamma \in \SL2(\BZ)} \frac{(yy')^{\frac{k}{2}}}{\left( \left( \frac{x-x'}{2} \right)^2 + \left( \frac{y+y'}{2} \right)^2  \right)^{\frac{k}{2}}} \le& 2 + 2 \sum_{\substack{c>0, \\ (c,d)=1}} \sum_{b \in \BZ} \frac{(yy'')^{\frac{k}{2}}}{\left( \left( \frac{x-x''-b}{2} \right)^2 + \left( \frac{y+y''}{2} \right)^2  \right)^{\frac{k}{2}}}\\
 &+ 2 \sum_{b > 0 }\frac{y^k}{\left( \left(\frac{b}{2}\right)^2+y^2\right)^{\frac{k}{2}}}
\end{aligned}$$
where $x''+iy''=\gamma_{c,d}z$. Recall the assumption $k>2$ for the Bergman kernel. For the last sum we have:
$$\begin{aligned}
\sum_{b>0} \frac{y^k}{\left( \left(\frac{b}{2}\right)^2+y^2\right)^{\frac{k}{2}}} &\le \sum_{0<b<2y+1} 1 + \sum_{b\ge 2y+1} \frac{y^k}{\left( \frac{b}{2}\right)^k} \\ 
& \le 2y+1 + \int_{2y}^{\infty} \frac{2^ky^k}{u^k} du \\
&\ll y.
\end{aligned}$$
For the inner sum of the middle sum we have:
$$\begin{aligned}
\sum_{b \in \BZ} \frac{(yy'')^{\frac{k}{2}}}{\left( \left( \frac{x-x''-b}{2} \right)^2 + \left( \frac{y+y''}{2} \right)^2  \right)^{\frac{k}{2}}} &\le  \sum_{|b| < y+y'' + 2} \frac{(yy'')^{\frac{k}{2}}}{\left( \frac{y+y''}{2} \right)^k} + 2\sum_{b \ge y+y''+1} \frac{(yy'')^{\frac{k}{2}}}{\left(  \frac{b}{2} \right)^k } \\
& \le \left(2(y+3)+1  \right)\frac{(yy'')^{\frac{k}{2}}}{\left(\frac{y+y''}{2} \right)^{k}} +  (yy'')^{\frac{k}{2}} \int_{y+y''}^{\infty} \frac{2^k}{u^k}du \\
& \ll y\frac{(yy'')^{\frac{k}{2}}}{\left(\frac{y+y''}{2} \right)^{k}}.
\end{aligned}$$
Summing this over the outer sum we get:
$$\begin{aligned}
\sum_{\substack{c>0, \\ (c,d)=1}} \frac{(yy'')^{\frac{k}{2}}}{\left(\frac{y+y''}{2} \right)^{k}} &= \sum_{\substack{c>0, \\ (c,d)=1}} \left( \frac{\displaystyle |cz+d|+\frac{1}{|cz+d|}}{2} \right)^{-k} \le  \sum_{c>0} \sum_{d \in \BZ} \left( \frac{\displaystyle |cz+d|+\frac{1}{|cz+d|}}{2} \right)^{-k} \\
& \le \sum_{c=1}^3 \sum_{d \in \BZ} \left( \frac{\displaystyle |cz+d|+\frac{1}{|cz+d|}}{2} \right)^{-k}+\sum_{c\ge 4} \sum_{d \in \BZ} \left( \frac{\displaystyle |cz+d|+\frac{1}{|cz+d|}}{2} \right)^{-k} \\
& \le \sum_{c=1}^3 \sum_{|d-cx|<2y+4 } \left( \frac{\displaystyle |cz+d|+\frac{1}{|cz+d|}}{2} \right)^{-k} +2\sum_{c=1}^3 \sum_{d \ge 2y+3} \left( \frac{d}{2} \right)^{-k} \\
 &\quad+\sum_{c\ge 4} \sum_{|d-cx|< 2cy+3} \left( \frac{cy}{2} \right)^{-k}+2\sum_{c\ge 4} \sum_{d \ge 2cy+1} \left( \frac{d}{2} \right)^{-k}\\ 
& \ll y+\int_{2y+2}^{\infty} \left(\frac{t}{2}\right)^{-k} dt + \int_{3}^{\infty} \left( \frac{sy}{2} \right)^{1-k} ds + \int_{3}^{\infty} \int_{2sy} \left( \frac{t}{2} \right)^{-k} dt ds \\
& \ll y\left(1+\frac{1}{k-2} \right).
\end{aligned}$$
So we have
\begin{equation}
\sum_{\gamma \in \SL2(\BZ)} \frac{(yy')^{\frac{k}{2}}}{\left( \left( \frac{x-x'}{2} \right)^2 + \left( \frac{y+y'}{2} \right)^2  \right)^{\frac{k}{2}}} \ll y\left(1+\frac{1}{k-2}\right).
\label{eq:Bergmantrivial}
\end{equation}
We now assume $y \ge 3, k\ge 6$, then we have:
\begin{equation}\begin{aligned}
\sum_{\gamma \in \SL2(\BZ)} \frac{(yy')^{\frac{k}{2}}}{\left( \left( \frac{x-x'}{2} \right)^2 + \left( \frac{y+y'}{2} \right)^2  \right)^{\frac{k}{2}}} &\le 2+ 8N \frac{y^k}{\left(\frac{1}{4}+y^2\right)^{\frac{k}{2}}} \\
&\quad  +  \sum_{\gamma \in \SL2(\BZ) \backslash \{\pm U^n|\ |n|\le 2N\}} \frac{(yy')^{\frac{k}{2}}}{\left( \left( \frac{x-x'}{2} \right)^2 + \left( \frac{y+y'}{2} \right)^2  \right)^{\frac{k}{2}}} \\
& \ll 1+ N \frac{y^k}{\left(\frac{1}{4}+y^2\right)^{\frac{k}{2}}} \\
&\quad  +  y \sup_{\gamma \in \SL2(\BZ) \backslash \{\pm U^n|\ |n|\le 2N\}} \frac{(yy')^{\frac{k-3}{2}}}{\left( \left( \frac{x-x'}{2} \right)^2 + \left( \frac{y+y'}{2} \right)^2  \right)^{\frac{k-3}{2}}}.
\label{eq:Bergmanreal}
\end{aligned}\end{equation}

We need to estimate this supremum. If $\gamma=\pm U^n$, then $n \ge 2N$ and we have:
$$
\frac{(yy')^{\frac{k-3}{2}}}{\left( \left( \frac{x-x'}{2} \right)^2 + \left( \frac{y+y'}{2} \right)^2  \right)^{\frac{k-3}{2}}} \le \frac{y^{k-3}}{\left( N^2 + y^2  \right)^{\frac{k-3}{2}}}= \left(1+\frac{N^2}{y^2} \right)^{-\frac{k-3}{2}}.
$$
Otherwise we have
$$\begin{aligned}
\frac{(yy')^{\frac{k-3}{2}}}{\left( \left( \frac{x-x'}{2} \right)^2 + \left( \frac{y+y'}{2} \right)^2  \right)^{\frac{k-3}{2}}} &\le \frac{(yy')^{\frac{k-3}{2}}}{\left( \frac{y+y'}{2}\right)^{k-3}} = \left(\frac{\displaystyle |cz+d| + \frac{1}{|cz+d|}}{2} \right)^{3-k} \\
&\le \left( \frac{y}{2}\right)^{3-k}.
\end{aligned}$$

We make the choice $N = \frac{y}{k^{\frac{1}{2}-\eta}}$ for some $\frac{1}{2}>\eta>0$, where we regard $\eta$ as a fixed constant. We thus have:
\begin{equation}
\sup_{\gamma \in \SL2(\BZ) \backslash \{\pm U^n|\ |n|\le 2N\}} \frac{(yy')^{\frac{k-3}{2}}}{\left( \left( \frac{x-x'}{2} \right)^2 + \left( \frac{y+y'}{2} \right)^2  \right)^{\frac{k-3}{2}}} \ll e^{-\delta k^{\frac{\eta}{2}}}
\label{eq:supremumfastdecay}
\end{equation}
for some $\delta>0$. We summarize these estimations in the following proposition.
\begin{prop} \label{prop:method2} Let $\nu$ be an automorphy factor of weight $k\ge 6$ for a finite index subgroup $\Gamma$ of $\SL2(\BZ)$, $\tau \in \SL2(\BZ)$ and $\{f_j\}$ an orthonormal basis of $S(\Gamma,k,\nu)$. Fix $\frac{1}{2}>\eta>0$ then we have for $z \in\BF_I$:
$$
y^k\sum_{j}|(f_j|_k \tau)(z)|^2 \ll_{\eta} \mu(\Gamma)k \left( 1+ \frac{y}{k^{\frac{1}{2}-\eta}} \right).
$$
\end{prop}
From this proposition we easily deduce Theorem \ref{thm:1}.\\

For the proof of theorems \ref{thm:2} and \ref{thm:3} we will use the Fourier expansion and a nice application of the Cauchy-Schwarz inequality to involve the Fourier coefficients of the Poincar\'e series:
$$\begin{aligned}
|(f_j|_k\tau)(z)|^2 &= \left| \sum_{m+\kappa_{\tau}>0} \widehat{(f_j|_k\tau)}(m) e^{\frac{2\pi i (m+\kappa_{\tau})z}{n_{\tau}}} \right|^2 \\
& \le \left( \sum_{m+\kappa_{\tau}>0} \widehat{|(f_j|_k\tau)}(m)| e^{-\frac{2\pi (m+\kappa_{\tau})y}{n_{\tau}}} \right)^2 \\
& \le \left( \sum_{m+\kappa_{\tau}>0} \lambda_m^{-1} \widehat{|(f_j|_k\tau)}(m)|^2 e^{-\frac{2\pi (m+\kappa_{\tau})y}{n_{\tau}}} \right) \left( \sum_{m+\kappa_{\tau}>0} \lambda_m e^{-\frac{2\pi (m+\kappa_{\tau})y}{n_{\tau}}} \right),
\end{aligned}$$
where the $\lambda_m$ are positive reals to be chosen later. Summing over $j$ we get:
\begin{equation} \begin{aligned}
y^k\sum_j |(f_j|_k \tau)(z)|^2 \le & \left( \sum_{m+\kappa_{\tau}>0} \lambda_m^{-1} A(m) y^{\frac{k}{2}} e^{-\frac{2\pi (m+\kappa_{\tau})y}{n_{\tau}}} \right) \\
&\times  \left( \sum_{m+\kappa_{\tau}>0} \lambda_m y^{\frac{k}{2}} e^{-\frac{2\pi (m+\kappa_{\tau})y}{n_{\tau}}} \right),
\label{eq:method1}
\end{aligned}\end{equation}
where
\begin{equation}\begin{aligned}
A(m) = & \frac{\mu(\Gamma)(4\pi(m+\kappa_{\tau}))^{k-1}}{n_{\tau}^k\Gamma(k-1)} \\
& \times  \left(1+2\pi i^{-k} \sum_{c=1}^{\infty} \frac{W(\Gamma^{\tau},\nu^{\tau};m,m;c)}{n_{\tau}c} J_{k-1}\left( \frac{4\pi(m+\kappa_{\tau})}{cn_{\tau}} \right) \right).
\end{aligned} \label{eq:funcA} \end{equation}
For the generalized Kloosterman sums we are going to use the trivial estimate:
$$
|W(\Gamma^{\tau},\nu^{\tau};m,m,c)|\le n_{\tau}^2c.
$$
Before we go any further we shall remark here that we can assume $k \gg 1$ as we want to investigate the behavior as $k \to \infty$.\\

We now have to deal with the sum
$$
\sum_{c=1}^{\infty} \left| J_{k-1}\left( \frac{4\pi(m+\kappa_{\tau})}{cn_{\tau}} \right) \right|.
$$

We split this into the different regions:
\begin{enumerate}
	\item $\sqrt{\frac{k-1}{2}} \ge \frac{4\pi(m+\kappa_{\tau})}{cn_{\tau}}$,
	\item $k-1-(k-1)^{\alpha} \ge \frac{4\pi(m+\kappa_{\tau})}{cn_{\tau}} \ge \sqrt{\frac{k-1}{2}}$ ,
	\item $k-1+(k-1)^{\alpha} \ge \frac{4\pi(m+\kappa_{\tau})}{cn_{\tau}} \ge k-1-(k-1)^{\alpha}$,
	\item $\frac{4\pi(m+\kappa_{\tau})}{cn_{\tau}} \ge k-1+(k-1)^{\alpha}$.
\end{enumerate}
Where $1\ge \alpha > \frac{1}{3}$ yet to be chosen.
For the first region we have by means of Proposition \ref{prop:JBesselverysmall}:
\begin{equation}\begin{aligned}
\sum_{c \ge 4 \pi \sqrt{\frac{2}{k-1}}\left( \frac{m+\kappa_{\tau}}{n_{\tau}} \right)} \left| J_{k-1}\left( \frac{4\pi(m+\kappa_{\tau})}{cn_{\tau}} \right) \right| & \le \frac{1}{\Gamma(k)} \sum_{c \ge 4 \pi \sqrt{\frac{2}{k-1}}\left( \frac{m+\kappa_{\tau}}{n_{\tau}} \right)} \left( \frac{2\pi(m+\kappa_{\tau})}{cn_{\tau}} \right)^{k-1} \\
& \le \frac{1}{\Gamma(k)} \left( \frac{2 \pi (m+\kappa_{\tau})}{n_{\tau}} \right)^{k-1} \left( \frac{\sqrt{2} \cdot 4 \pi(m+\kappa_{\tau})}{\sqrt{k-1}n_{\tau}} \right)^{k-1} \\
& \quad \quad \times \left( 1+  \frac{\sqrt{2} \cdot 4 \pi(m+\kappa_{\tau})}{\sqrt{k-1}(k-2)n_{\tau}}\right) \\
& \ll \frac{1}{\Gamma(k)} \left( \frac{k-1}{8}\right)^{\frac{k-1}{2}} \left(1+\frac{m+\kappa_{\tau}}{n_{\tau}}\cdot k^{-\frac{3}{2}} \right) \\
& \ll k^{-\frac{k}{2}} \left(1+\frac{m+\kappa_{\tau}}{n_{\tau}}\cdot k^{-\frac{3}{2}} \right).
\label{eq:region1}
\end{aligned} \end{equation}
In the second region we have by Proposition \ref{prop:JBesselsmall}:
\begin{equation}\begin{aligned}
\sum_{\substack{4 \pi \sqrt{\frac{2}{k-1}}\left( \frac{m+\kappa_{\tau}}{n_{\tau}} \right) \ge c, \\ c \ge 4\pi\left( \frac{m+\kappa_{\tau}}{n_{\tau}} \right)(k-1-(k-1)^{\alpha})^{-1} }} \left| J_{k-1}\left( \frac{4\pi(m+\kappa_{\tau})}{cn_{\tau}} \right) \right| & \ll \left( \frac{m+\kappa_{\tau}}{n_{\tau}} \right)k^{-\frac{1}{2}} \cdot k^{-\frac{4}{3}} \\
& \ll \left( \frac{m+\kappa_{\tau}}{n_{\tau}} \right)k^{-\frac{11}{6}}.
\label{eq:region2}
\end{aligned}\end{equation}

For the third region we have using Proposition \ref{prop:JBesselmed}:

\begin{equation}\begin{aligned}
\sum_{\substack{4\pi\left( \frac{m+\kappa_{\tau}}{n_{\tau}} \right)(k-1-(k-1)^{\alpha})^{-1} \ge c, \\ c \ge  4\pi\left( \frac{m+\kappa_{\tau}}{n_{\tau}} \right)(k-1+(k-1)^{\alpha})^{-1}}} \left| J_{k-1}\left( \frac{4\pi(m+\kappa_{\tau})}{cn_{\tau}} \right) \right| & \ll \left( \frac{m+\kappa_{\tau}}{n_{\tau}} \right) k^{\alpha-2} \cdot k^{-\frac{1}{3}} \\
& \ll \left( \frac{m+\kappa_{\tau}}{n_{\tau}} \right) k^{\alpha-\frac{7}{3}}.
\label{eq:region3}
\end{aligned}\end{equation}

And in the last region we have by Proposition \ref{prop:JBessellarge}:

\begin{equation}\begin{aligned}
\sum_{\substack{4\pi\left( \frac{m+\kappa_{\tau}}{n_{\tau}} \right)(k-1+(k-1)^{\alpha})^{-1} \ge c}} \left| J_{k-1}\left( \frac{4\pi(m+\kappa_{\tau})}{cn_{\tau}} \right) \right| & \ll \left( \frac{m+\kappa_{\tau}}{n_{\tau}} \right) k^{-1} \cdot k^{-\frac{\alpha+1}{4}} \\
& \ll \left( \frac{m+\kappa_{\tau}}{n_{\tau}} \right) k^{-\frac{\alpha+5}{4}}.
\label{eq:region4}
\end{aligned}\end{equation}

We make the choice $\alpha= \frac{13}{15}$ and get for $A(m)$ (defined by equation \eqref{eq:funcA}) the estimation:
\begin{equation}
|A(m)| \ll \frac{\mu(\Gamma)(4\pi)^k}{n_{\tau}^k\Gamma(k-1)} \left((m+\kappa_{\tau})^{k-1}(1+n_{\tau}k^{-\frac{k}{2}})+(m+\kappa_{\tau})^{k}k^{-\frac{22}{15}} \right).
\label{eq:Aest}
\end{equation}

Considering the inequality \eqref{eq:method1} and the Cauchy-Schwartz equality case we should choose $\lambda_m \approx (m+\kappa_{\tau})^{\frac{k}{2}}$. So lets put $\lambda_m=(m+\kappa_{\tau})^{\frac{k}{2}+\delta}$ with $\delta=o(k)$.\\
The sum
\begin{equation}
S(\alpha,\beta,\eta)=\sum_{m+\eta>0} (m+\eta)^{\alpha}e^{-\beta(m+\eta)}, \quad \alpha,\beta,\eta>0
\label{eq:Sab}
\end{equation}
appears often in the next few calculations, hence the following lemma will be useful.

\begin{lem} \label{lem:Sabest} $S(\alpha,\beta,\kappa)$ as defined by \eqref{eq:Sab} satisfies the following inequalities:
$$
S(\alpha,\beta,\eta) \le \beta^{-\alpha-1}\Gamma(\alpha+1)+\beta^{-\alpha} \alpha^{\alpha} e^{-\alpha}
$$
and for $\alpha \le \beta \eta$ we have:
$$
S(\alpha,\beta,\eta) \le \beta^{-\alpha-1}\Gamma(\alpha+1) + \eta^{\alpha} e^{-\beta\eta}.
$$
\end{lem}
\begin{proof} The function $x^{\alpha}e^{-\beta x}$ increases on $(0,\frac{\alpha}{\beta}]$ and decreases on $[\frac{\alpha}{\beta},\infty)$. Hence we get
$$\begin{aligned}
S(\alpha,\beta,\eta) & \le \int_{\eta}^{\infty} x^{\alpha}e^{-\beta x} dx + \left( \frac{\alpha}{\beta} \right)^{\alpha} e^{-\beta\frac{\alpha}{\beta}} \\
& \le \int_{0}^{\infty} x^{\alpha}e^{-\beta x} dx +  \beta^{-\alpha} \alpha^{\alpha} e^{-\alpha}\\
& = \beta^{-\alpha-1}\Gamma(\alpha+1)+\beta^{-\alpha} \alpha^{\alpha} e^{-\alpha}.
\end{aligned}$$

And if one assumes $\alpha \le \beta \eta$, then:
$$\begin{aligned}
S(\alpha,\beta,\eta) & \le \int_{\eta}^{\infty} x^{\alpha}e^{-\beta x} dx + \eta^{\alpha}e^{-\beta \eta} \\
& \le \beta^{-\alpha-1}\Gamma(\alpha+1)+ \eta^{\alpha}e^{-\beta \eta}.
\end{aligned}$$

\end{proof}

Using \eqref{eq:Aest} in \eqref{eq:method1} with the choice $\lambda_m=(m+\kappa_{\tau})^{\frac{k}{2}+\delta}$ we get:
\begin{equation*}\begin{aligned}
y^k\sum_j |(f_j|_k \tau)(z)|^2  & \le \frac{y^k (4\pi)^k \mu(\Gamma)}{n_{\tau}^k\Gamma(k-1)} S\left(\frac{k}{2}+\delta,\frac{2\pi y}{n_{\tau}},\eta_{\tau}\right) \\
& \quad \times \left( (1+n_{\tau}k^{-\frac{k}{2}})S\left(\frac{k}{2}-\delta-1,\frac{2 \pi y}{n_{\tau}},\eta_{\tau}\right)+ k^{-\frac{22}{15}} S \left(\frac{k}{2}-\delta,\frac{2 \pi y}{n_{\tau}}, \eta_{\tau} \right)\right).
\label{eq:method1low}
\end{aligned}\end{equation*}

Using Lemma \ref{lem:Sabest} we have:
$$\begin{aligned}
S\left(\frac{k}{2}+\delta,\frac{2\pi y}{n_{\tau}},\eta_{\tau}\right) & \le \left(\frac{2 \pi y}{n_{\tau}}\right)^{-\frac{k}{2}-\delta-1} \Gamma\left(\frac{k}{2}+\delta+1\right) + \left( \frac{2 \pi y}{n_{\tau}} \right)^{-\frac{k}{2}-\delta} \left(\frac{k}{2}+\delta\right)^{\frac{k}{2}+\delta}e^{-\frac{k}{2}-\delta} \\
& \ll \left(\frac{2 \pi y}{n_{\tau}}\right)^{-\frac{k}{2}-\delta-1} \left( \frac{k}{2} \right)^{\frac{k+1}{2}+\delta} e^{-\frac{k}{2}-\delta} \left[ e^{\delta} + \frac{2 \pi y}{n_{\tau}}k^{-\frac{1}{2}}e^{\delta}\right]\\
& \ll \frac{(4 \pi)^{-\frac{k}{2}-\delta}y^{-\frac{k}{2}-\delta-1}k^{\frac{k+1}{2}+\delta}e^{-\frac{k}{2}}}{n_{\tau}^{-\frac{k}{2}-\delta -1}} \left[ 1+\frac{yk^{-\frac{1}{2}}}{n_{\tau}} \right], \\
S\left(\frac{k}{2}-\delta-1,\frac{2 \pi y}{n_{\tau}},\eta_{\tau}\right) 
& \ll \frac{(4 \pi)^{-\frac{k}{2}+\delta}y^{-\frac{k}{2}+\delta}k^{\frac{k-1}{2}-\delta}e^{-\frac{k}{2}}}{n_{\tau}^{-\frac{k}{2}+\delta}} \left[1+\frac{yk^{-\frac{1}{2}}}{n_{\tau}} \right], \\
S \left(\frac{k}{2}-\delta,\frac{2 \pi y}{n_{\tau}}, \eta_{\tau} \right) 
& \ll \frac{(4 \pi)^{-\frac{k}{2}+\delta}y^{-\frac{k}{2}+\delta-1}k^{\frac{k+1}{2}-\delta}e^{-\frac{k}{2}}}{n_{\tau}^{-\frac{k}{2}+\delta -1}} \left[ 1+\frac{yk^{-\frac{1}{2}}}{n_{\tau}} \right].
\end{aligned}$$
Plugging these inequalities into \eqref{eq:method1low} we get:
\begin{prop} \label{prop:method1low} Let $\nu$ be an automorphy factor of weight $k \gg 1$ for $\Gamma$ a finite index subgroup of $\SL2(\BZ)$, $\tau \in \SL2(\BZ)$ and $\{f_j\}$ an orthonormal basis of $S(\Gamma,k,\nu)$ then we have for $z \in \BF_I$:
\begin{equation}\begin{aligned}
y^k\sum_j |(f_j|_k \tau)(z)|^2  & \ll \frac{\mu(\Gamma)n_{\tau}k^{\frac{3}{2}}}{y} \left[1+\frac{yk^{-\frac{1}{2}}}{n_{\tau}} \right]^2 \left[ 1+n_{\tau}k^{-\frac{k}{2}}+\frac{n_{\tau}}{y}k^{-\frac{7}{15}} \right].
\label{eq:method1lowfinal}
\end{aligned}\end{equation}
\end{prop}

For large $y$ we can improve on this. For this purpose we assume $|\delta|+1\le \frac{k}{2}$ and $y \ge \frac{3n_{\tau}k}{\eta_{\tau}\pi}$. These assumptions will allow us to use the following lemma.

\begin{lem} \label{lem:expdecay} The following inequality holds for $x \ge 6 \frac{\alpha}{\beta}, \ \alpha,\beta>0$:
$$
x^{\alpha}e^{-\beta x} \le \alpha^{\alpha}\beta^{-\alpha}e^{-\alpha} \cdot e^{-\frac{\beta x}{2}}.
$$
\end{lem}
\begin{proof} Let $x=c \frac{\alpha}{\beta}$, then
$$
x^{\alpha}e^{-\beta x} = \alpha^{\alpha}\beta^{-\alpha}e^{-\alpha} \cdot e^{-\frac{\beta x}{2}} \cdot \left(c e^{1-\frac{c}{2}}\right)^{\alpha}.
$$
Note that $6e^{1-3}<1$ and that $ce^{1-\frac{c}{2}}$ is decreasing on $[2,\infty)$.
\end{proof}

Using Lemmata \ref{lem:Sabest} and \ref{lem:expdecay} we get:

$$\begin{aligned}
S\left(\frac{k}{2}+\delta,\frac{2\pi y}{n_{\tau}},\eta_{\tau}\right) & \le \left(\frac{2 \pi y}{n_{\tau}}\right)^{-\frac{k}{2}-\delta-1} \Gamma\left(\frac{k}{2}+\delta+1\right) + \eta_{\tau}^{\frac{k}{2}+\delta}e^{-\frac{2\pi y}{n_{\tau}}\eta_{\tau}} \\
& \ll \left(\frac{2 \pi y}{n_{\tau}}\right)^{-\frac{k}{2}-\delta-1} \left( \frac{k}{2} \right)^{\frac{k+1}{2}+\delta} e^{-\frac{k}{2}} \\
& \quad \times \left[ 1 + \left(\frac{2 \pi y}{n_{\tau}}\right)^{\frac{k}{2}+\delta+1} \left( \frac{k}{2} \right)^{-\frac{k+1}{2}-\delta}e^{\frac{k}{2}}   \eta_{\tau}^{\frac{k}{2}+\delta}e^{-\frac{2\pi y}{n_{\tau}}\eta_{\tau}}\right] \\
& \ll \left(\frac{2 \pi y}{n_{\tau}}\right)^{-\frac{k}{2}-\delta-1} \left( \frac{k}{2} \right)^{\frac{k+1}{2}+\delta} e^{-\frac{k}{2}} \left[ 1 + \eta_{\tau}^{-1}k^{\frac{1}{2}} e^{-\frac{\eta_{\tau}\pi}{n_{\tau}} y }\right], \\
S\left(\frac{k}{2}-\delta-1,\frac{2 \pi y}{n_{\tau}},\eta_{\tau}\right) & \ll \left(\frac{2 \pi y}{n_{\tau}}\right)^{-\frac{k}{2}+\delta} \left( \frac{k}{2} \right)^{\frac{k+1}{2}-\delta-1} e^{-\frac{k}{2}} \left[ 1 + \eta_{\tau}^{-1}k^{\frac{1}{2}} e^{-\frac{\eta_{\tau}\pi}{n_{\tau}} y }\right], \\
S \left(\frac{k}{2}-\delta,\frac{2 \pi y}{n_{\tau}}, \eta_{\tau} \right) & \ll \left(\frac{2 \pi y}{n_{\tau}}\right)^{-\frac{k}{2}+\delta-1} \left( \frac{k}{2} \right)^{\frac{k+1}{2}-\delta} e^{-\frac{k}{2}} \left[ 1 + \eta_{\tau}^{-1}k^{\frac{1}{2}} e^{-\frac{\eta_{\tau}\pi}{n_{\tau}} y }\right].
\end{aligned}$$

Plugging these inequalities into \eqref{eq:method1low} we get:
\begin{prop} \label{prop:method1large} Let $\nu$ be an automorphy factor of weight $k \gg 1$ for $\Gamma$ a finite index subgroup of $\SL2(\BZ)$, $\tau \in \SL2(\BZ)$ and $\{f_j\}$ an orthonormal basis of $S(\Gamma,k,\nu)$ then we have for $z \in\BF_I, y \ge \frac{3n_{\tau}k}{\eta_{\tau}\pi}$:
\begin{equation}\begin{aligned}
y^k\sum_j |(f_j|_k \tau)(z)|^2  & \ll \frac{\mu(\Gamma)n_{\tau}k^{\frac{3}{2}}}{y} \left[ 1 + \eta_{\tau}^{-1}k^{\frac{1}{2}} e^{-\frac{\eta_{\tau}\pi}{n_{\tau}} y }\right]^2 \left[ 1+n_{\tau}k^{-\frac{k}{2}}+\frac{n_{\tau}}{y}k^{-\frac{7}{15}} \right].
\label{eq:method1largefinal}
\end{aligned}\end{equation}
\end{prop}

For the proof of Theorem \ref{thm:2} we split into different areas. For $y \le \max_{\tau \in \SL2(\BZ)}n_{\tau}$ we use Proposition \ref{prop:method2}, for $y \ge \frac{3 n_{\tau}k}{\eta_{\tau}\pi}$ we use Proposition \ref{prop:method1large} and for $y$ in between we use Proposition \ref{prop:method1low}.\\

For the proof of Theorem \ref{thm:3} we are left to prove a lower bound. We have:
\begin{equation}\begin{aligned}
y^{\frac{k}{2}} \left|\widehat{(f|_k\tau)}(m) \right| &= \frac{1}{n_{\tau}} \left| \int_0^{n_{\tau}} y^{\frac{k}{2}}(f|_k \tau)(z) e^{-\frac{2 \pi i(m+\kappa_{\tau})z}{n_{\tau}}} dx\right| \\
& \le \frac{1}{n_{\tau}} e^{\frac{2 \pi (m + \kappa_{\tau})}{n_{\tau}}y} \cdot  \int_0^{n_{\tau}} y^{\frac{k}{2}} \left| (f|_k \tau)(z) \right|dx.
\label{eq:fourierlow}
\end{aligned}\end{equation}
If we sum the squares of this inequality over an orthonormal basis $\{f_j\}$ we can use the Fourier coefficients of the Pioncar\'e series (see Corollary \ref{cor:Fouriersqr}) for the left hand side and for the right hand side we can use Cauchy-Schwarz to get:
\begin{equation}\begin{aligned}
\sup_{\im z = y} \sum_j y^k|(f_j|_k \tau)(z)|^2 & \ge \frac{1}{n_{\tau}^2} \sum_j \left ( \int_0^{n_{\tau}} dx\right) \left( \int_0^{n_{\tau}} y^k|(f_j|_k \tau)(z)|^2 dx \right ) \\
& \ge \frac{1}{n_{\tau}^2} \sum_j \left( \int_0^{n_{\tau}} y^{\frac{k}{2}} \left| (f|_k \tau)(z) \right|  dx\right)^2\\
& \ge y^k e^{-\frac{4 \pi (m+\kappa_{\tau})}{n_{\tau}}y} \cdot \sum_j \left|\widehat{(f|_k\tau)}(m) \right|^2 \\
& \ge y^k e^{-\frac{4 \pi (m+\kappa_{\tau})}{n_{\tau}}y} \cdot \frac{\mu(\Gamma) (4 \pi (m+\kappa_{\tau}))^{k-1} }{n_{\tau}^k \Gamma(k-1)} \\
& \quad \times \left( 1 +2\pi i^{-k} \sum_{c=1}^{\infty} \frac{W(\Gamma^{\tau},\nu^{\tau};m,m;c)}{n_{\tau}c} J_{k-1}\left( \frac{4\pi(m+\kappa_{\tau})}{cn_{\tau}} \right) \right).
\label{eq:lowavgbound}
\end{aligned}\end{equation}
For $k \ge 320(m+1)^2$ we can use Proposition \ref{prop:JBesselverysmall} to estimate that the right hand side is bigger or equal to
$$
y^k e^{-\frac{4 \pi (m+\kappa_{\tau})}{n_{\tau}}y} \cdot \frac{\mu(\Gamma) (4 \pi (m+\kappa_{\tau}))^{k-1} }{n_{\tau}^k \Gamma(k-1)} \left( 1 - 2\pi n_{\tau} \frac{\zeta(k-1)}{\Gamma(k)} \left( \frac{2 \pi (m+\kappa_{\tau})}{n_{\tau}} \right)^{k-1} \right).
$$
We get our desired lower bound if we takes $y=\frac{k n_{\tau}}{4 \pi (m+ \kappa_{\tau})}$ and $m,\tau$ such that $m+\kappa_{\tau}$ is minimal, which is $\min_{\tau \in \SL2(\BZ)} \eta_{\tau}$.
\begin{acknowledgement}This work is based on the author's master's thesis, which he completed during November 2013 - April 2014 in Bristol, UK. I would like to thank my two supervisors professor Emmanuel Kowalski for enabling me to do my master's thesis abroad in Bristol, and Dr. Abhishek Saha for introducing me to the sup-norm problem of modular forms as well as giving helpful comments and guidance.
\end{acknowledgement}
\newpage

\bibliography{Bibliography}

\begin{thebibliography}{EMOT81}

\bibitem[DS13]{DasSeng}
Soumya Das and Jyoti Sengupta.
\newblock {$L^\infty$} norms of holomorphic modular forms in the case of
  compact quotient.
\newblock {\em {P}reprint}, 2013.
\newblock {\tt arXiv:1301.3677}.

\bibitem[EMOT81]{HTF}
Arthur Erd{\'e}lyi, Wilhelm Magnus, Fritz Oberhettinger, and Francesco~G.
  Tricomi.
\newblock {\em Higher transcendental functions. {V}ol. {II}}.
\newblock Robert E. Krieger Publishing Co., Inc., Melbourne, Fla., 1981.
\newblock Based on notes left by Harry Bateman, Reprint of the 1953 original.

\bibitem[FJK13]{FJK}
Joshua~S Friedman, Jay Jorgenson, and Jurg Kramer.
\newblock Uniform sup-norm bounds on average for cusp forms of higher weights.
\newblock {\em {P}reprint}, 2013.
\newblock {\tt arXiv:1305.1348}.

\bibitem[HT12]{HT2}
Gergely Harcos and Nicolas Templier.
\newblock On the sup-norm of {M}aass cusp forms of large level: {II}.
\newblock {\em Int. Math. Res. Not. IMRN}, (20):4764--4774, 2012.

\bibitem[HT13]{HT3}
Gergely Harcos and Nicolas Templier.
\newblock On the sup-norm of {M}aass cusp forms of large level. {III}.
\newblock {\em Math. Ann.}, 356(1):209--216, 2013.

\bibitem[IS95]{IS95}
H.~Iwaniec and P.~Sarnak.
\newblock {$L^\infty$} norms of eigenfunctions of arithmetic surfaces.
\newblock {\em Ann. of Math. (2)}, 141(2):301--320, 1995.

\bibitem[Kir13]{halflevel}
Eren~Mehmet Kiral.
\newblock Bounds on sup-norms of half-integral weight modular forms.
\newblock {\em {P}reprint}, 2013.
\newblock {\tt arXiv:1309.7218}.

\bibitem[Ran77]{MFaF}
Robert~A. Rankin.
\newblock {\em Modular forms and functions}.
\newblock Cambridge University Press, Cambridge-New York-Melbourne, 1977.

\bibitem[Rud05]{R05}
Ze{\'e}v Rudnick.
\newblock On the asymptotic distribution of zeros of modular forms.
\newblock {\em Int. Math. Res. Not.}, (34):2059--2074, 2005.

\bibitem[Sah14]{Saha14}
Abhishek Saha.
\newblock On sup-norms of cusp forms of powerful level.
\newblock {\em {P}reprint}, 2014.
\newblock {\tt arXiv:1404.3179}.

\bibitem[Tem11]{Thybrid}
NICOLAS Templier.
\newblock Hybrid sup-norm bounds for hecke-maass cusp forms.
\newblock {\em To appear J. Eur. Math. Soc}, 2011.

\bibitem[Wat44]{ToBF}
G.~N. Watson.
\newblock {\em A {T}reatise on the {T}heory of {B}essel {F}unctions}.
\newblock Cambridge University Press, Cambridge, England; The Macmillan
  Company, New York, 1944.

\bibitem[Xia07]{Supnormintweight}
Honggang Xia.
\newblock On {$L^\infty$} norms of holomorphic cusp forms.
\newblock {\em J. Number Theory}, 124(2):325--327, 2007.

\end{thebibliography}
\end{document}